\documentclass[11pt]{article}
\usepackage{graphicx}
\usepackage{csquotes}
\usepackage{epsfig}
\usepackage{lscape}
\usepackage{amsfonts}
\usepackage[misc,geometry]{ifsym}
\usepackage{color}
\usepackage{amssymb,amsmath}
\usepackage{listings}
\usepackage{rotating}
\usepackage{graphicx}
\usepackage{multicol}
\usepackage[english]{babel}
\usepackage{lineno}
\usepackage{mathrsfs}
\usepackage[pdfpagelabels=true,plainpages=false,colorlinks=true,linkcolor=blue,citecolor=blue,urlcolor=blue]{hyperref}
\date{}
\makeatletter \oddsidemargin  -0.5cm \evensidemargin -0.5cm \textwidth
16.5cm \topmargin 0.005cm \textheight 22cm
\usepackage{amssymb}
\usepackage{amsmath}
\usepackage{amsthm}
\usepackage{enumerate}
\usepackage{float}
\restylefloat{table}
\usepackage{caption}
\usepackage{mathtools}
\usepackage{lipsum}
\usepackage{xcolor}
\usepackage{soul}
\sethlcolor{black}
\makeatletter
\let\chapter\section
\usepackage[ruled,vlined]{algorithm2e}
\usepackage{afterpage}
\usepackage{url}
\usepackage{geometry}
\usepackage{pdflscape}
\usepackage{cases}
\usepackage{rotating}
\usepackage{booktabs}

\usepackage{tikz}

\usepackage[numbers,sort&compress]{natbib}
\bibpunct[, ]{[}{]}{,}{n}{,}{,}
\makeatletter
\def\NAT@def@citea{\def\@citea{\NAT@separator}}
\makeatother
\theoremstyle{plain}
\newtheorem{theorem}{Theorem}[section]

\newtheorem{corollary}[theorem]{Corollary}

\usepackage{placeins}
\usepackage[ruled,vlined]{algorithm2e}
\newtheorem{defn}{Definition}
\newtheorem{prop}{Proposition}

\newtheorem{note}{Note}
\newtheorem{ex}{Example}

\usepackage{mathrsfs}
\usepackage{epsfig}

\begin{document}

\markboth{P.Roy, G.Panda}
{Expansion of interval valued function}


\title{Expansion of Generalized Hukuhara Differentiable Interval Valued Function}

\author{Priyanka Roy$ ^\ddagger $  and Geetanjali Panda
\\
 Department of Mathematics, Indian Institute of Technology Kharagpur, \\
Kharagpur, West Bengal-721302, India.\\
$ ^\ddagger $ proy180192@maths.iitkgp.ac.in}

\maketitle


\begin{abstract}
In this article the concept of $\mu-$ monotonic property of interval valued function in higher dimension is introduced. Expansion of interval valued function in higher dimension is developed using this property. Generalized Hukuhara differentiability is used to derive the theoretical results. Several examples are provided to justify the theoretical developments.\\
\textbf{keywords} Interval function; $\mu$ monotonic property; generalized Hukuhara derivative.
\end{abstract}

\section{Introduction}
Importance of the study of uncertainty theory from theoretical point of view has been increased in recent years due to its application in several issues of image processing, control theory, decision making, dynamic economy, optimization theory etc. Due to the increasing in complexity of environment, change of climate and inherent nature of human thought, crisp values are insufficient to make real life decision making problems. In these uncertain environments, parameters of the mathematical models are accepted as uncertain, which are usually considered in linguistic sense or in probabilistic sense. However, it is not always convenient to build appropriate membership function and probability distribution function to handle the linguistic parameters and probabilistic parameters respectively. To avoid this difficulty, in recent times, the uncertain parameters are considered as intervals, where the upper and lower bounds of the parameters are estimated from the historical data. In that case, the functions involved in the model have bounded parameters and known as interval valued functions. Interval analysis plays an important role to handle these functions. Calculus of set valued function is based on generalized Hukuhara difference(gH-difference) which is explored in Refs.\cite{aubindifferential, chalco2011generalized, de2007existence, galanis2005set, ibrahim1996differentiability, li2009calculus, tu2009stability}. Since the interval valued function is a particular case of set valued function, so  gH difference ($ \ominus_{gH}$) is defined for two intervals and used in  uncertainty theory including interval analysis, fuzzy set theory, interval optimization, interval differential equations etc.(see Refs. \cite{bede2013generalized, chalco2013calculus, costa2015generalized, Lupulescu201350, Malinowski2011JAML, stefanini2008generalization, stefanini2009generalized}).
So far, calculus of interval valued function is widely studied and applied in different types of mathematical models, but expansion of interval valued function remains an untouched area of research. The present contribution has addressed this gap to some extent. Rall \cite{rall1983mean} developed interval version of mean value theorem and Taylor's theorem using interval inclusion property and Gateaux type derivative. In this article, generalized Hukuhara difference is used to study the expansion of interval valued functions from $ \mathbb{R}^n $ to the set of intervals with the help of $ \mu $ monotonic property.
 \par
        An interval valued function  $\hat{f}$ may be  treated either as the image extension of a real valued function $f: \mathbb{R}^n\rightarrow \mathbb{R}$, represented by  $\hat{f}(\hat{A})=\{f(x): x\in \hat{A} \mbox{, $\hat{A}$ is a closed interval vector}\}$ or as a function from $\mathbb{R}^n$ to the set of intervals, whose parameters are intervals and arguments are real. For example, image extension of a real valued function $f(x_1,x_2)=2x_1+3x_2$ over an interval vector $(X_1,X_2)=([1,3], [0,2])$ is $\hat{f}(X_1,X_2)=2[1,3]+3[0,2]$, where as an example of the second category interval function may be $\hat{f}(x_1,x_2)=[1,4]x_1^2+[0,1]x_2$. Several numerical algorithms are designed using the concept of image extension of real valued functions to compute the rigorous bounds of  approximate errors while solving system of equations, determining the bounds for exact value of integrals, and other scientific computations. For the existing literature in this area the readers may see Refs.\cite{rall1983mean, caprani1980mean, hansen2003global, Baker1996kap, moore2009introduction, Neumaier1990book}. This article has focused on the second type interval valued functions ($\hat{f}$ from $R^n$ to the set of closed intervals), whose arguments are real variables and parameters are intervals.
\\\\
Contribution of the paper is explored in different sections. Some notations and preliminaries on interval analysis are discussed in Section 2. $\mu-$ monotonic property of interval valued function of single variable is developed in the existing theory of interval analysis \cite{markov1979calculus}. Using this concept, $\mu-$ monotonicity of interval valued function over $\mathbb{R}^n$ is introduced in Section 3 and  calculus of interval valued function over $\mathbb{R}^n$ is revisited. In Section 4, expansion  of interval valued functions in higher dimension is developed using the concept of previous section. Numerical examples are provided for the justification of the theoretical developments.
\section{Some notations and preliminaries}
Let $ I(\mathbb{R}) $ be the set of all closed intervals on the real line $ \mathbb{R} $. $ \hat{a} \in I(\mathbb{R}) $ is the closed interval of the form $ [\underline{a},\overline{a}] $ where $ \underline{a} \leq \overline{a} $. Spread of the interval $ \hat{a} $ is denoted by $ \mu(\hat{a})$, where $ \mu(\hat{a})=\overline{a} - \underline{a} $. For two points $ a_{1}$ and $ a_{2} $,(not necessarily $ a_{1}\leq a_{2}  $), $ \hat{a} $ can be written as $  \hat{a}=[a_{1} \vee a_{2}]  =\left[\min\left\lbrace a_{1},a_{2}\right\rbrace ,\max\left\lbrace  a_{1},a_{2}\right\rbrace  \right]  $. Any real number $x$ can be expressed as a degenerate interval denoted by $\hat{x}$,  $\hat{x}=[x,x] $~or $x.\hat{I}$, where $\hat{I}=[1,1]$. $\hat{0}=[0,0]=0.\hat{I}$ denotes the null interval.\\
    Algebraic  operation  between two intervals  $ \hat{a} $, $ \hat{b} $  is defined as  $  \hat{a} \circledast \hat{b}=\left\lbrace a* b | a\in \hat{a},b\in \hat{b}  \right\rbrace$, where $ * \in \left\lbrace +,-,\cdot,/ \right\rbrace  $. Additive inverse in $ \left\langle I(\mathbb{R}),\oplus,\odot\right\rangle  $ may not exist, that is, $ \hat{a}\ominus \hat{a} $ is not necessarily $\hat{0}$ according to this approach.
To overcome this difficulty, the $ gH $ difference between two intervals is defined by \textbf{L. Stefanini} \cite{stefanini2009generalized}. For $ \hat{a}, \hat{b} \in I(\mathbb{R}) $,
$$    \hat{a} \ominus_{gH} \hat{b} = \left[ \min \left\lbrace \underline{a}- \underline{b}, \overline{a}-\overline{b}          \right\rbrace , \max \left\lbrace \underline{a}- \underline{b}, \overline{a}-\overline{b}          \right\rbrace  \right] .                  $$
This is the most generalized concept of interval difference used in interval calculus so far. As per this difference, for the intervals $\hat{a},\hat{b}$ and $\hat{c}$, $ \ominus_{gH} \hat{a}= \hat{0}\ominus_{gH} \hat{a}= (-1)\odot \hat{a} $ and
\begin{align}\hat{a}\ominus_{gH}\hat{b}=\hat{c} \Leftrightarrow \begin{cases} \hat{a}=\hat{b}\oplus \hat{c} \\
\mbox{~ or~~}\\
\hat{a}\oplus(-1)\hat{c}=\hat{b}\end{cases}\label{dif}\end{align}

Product of an interval with a real number, product of an interval vector with a real vector and product of a real matrix with an interval vector are defined as follows which are used throughout the article.
\begin{enumerate}
\item For $ a \in \mathbb{R} $, $ \hat{b}=[\underline{b}, \overline{b}] \in I(\mathbb{R}) $, $ a \hat{b}= \begin{cases}
[a\underline{b}, a\overline{b}] \text{  if }  a \geq 0\\
[a \overline{b}, a\underline{b}] \text{  if }  a \leq 0
\end{cases} $.
\item For $ p= (p_1,~ p_2,~\cdots, ~p_n)^T \in \mathbb{R}^n $ and  $ \hat{q}= (\hat{q}_1, ~ \hat{q}_2,~\cdots, ~\hat{q}_n)^T \in (I(\mathbb{R}))^n $, $ p^T\hat{q}= \sum_{i=1}^n p_i \hat{q}_i $.
\item For a real matrix $ A= (a_{ij})_{n \times m} \in \mathbb{R}^{n \times m} $ and an interval vector $ \hat{q}= (\hat{q}_1, ~ \hat{q}_2,~\cdots,~ \hat{q}_n)^T \in (I(\mathbb{R}))^n $,\\ $ A^{T} \hat{q} = \begin{pmatrix}
\sum_{i=1}^n a_{i1} \hat{q}_i, & \sum_{i=1}^n a_{i2} \hat{q}_i, & \cdots, & \sum_{i=1}^n a_{im} \hat{q}_i
\end{pmatrix}^T$
\end{enumerate}
An interval valued function  $\hat{f}: \mathbb{R}^n\rightarrow I(\mathbb{R}) $ can be expressed in the form   $\hat{f}(x)=[\underline{f}(x),\overline{f}(x)] $, where  $ \underline{f}(x)\leq \overline{f}(x) $, $ \forall~~x\in \mathbb{R}^n$,  $ \underline{f}, \overline{f}: \mathbb{R}^n\rightarrow \mathbb{R}$.
\\
Spread of $\hat{f}(x)$ is denoted by $\mu_{\hat{f}}(x)\triangleq\overline{f}(x)-\underline{f}(x)$.\\
\\\\
Some existing results on $ gH $-differentiability( which are based on  $ gH $-difference) are provided in this section for the interval valued function  
 $ \hat{f}$ on $\mathbb{R}$ and $\mathbb{R}^n$.
 Limit and continuity of an interval valued function are understood in the sense of metric structure of $ gH $ difference using Hausdorff distance between intervals as discussed in Ref. \cite{stefanini2009generalized}.
\begin{defn}[Stefanini and Bede \cite{stefanini2009generalized}]\label{Definition 2}
 Generalized Hukuhara derivative of $ \hat{f}: (t_1,t_2)\subseteq \mathbb{R} \rightarrow I(\mathbb{R}) $ at $ x_{0}\in (t_1,t_2) $ is defined as
$ \hat{f}^{\prime}(x_{0})= \lim_{h\rightarrow 0}\frac{\hat{f}(x_{0}+h) \ominus_{gH} \hat{f}(x_{0})}{h}$.
\end{defn}
\begin{defn}[R~Osuna-G{\'o}mez et. al\cite{osuna2015optimality}]
\label{Partial derivatives}
For an interval valued function $ \hat{f}:\mathbb{R}^{n}\rightarrow I(\mathbb{R}) $, if $ \lim_{h_i\rightarrow 0} \frac{1}{h_i}\left( \hat{f}(x_1,x_2, \cdots x_{i}+h_i, \cdots x_n)\ominus_{gH}\hat{f}(x)\right) $ exists, then we say that the partial derivative of $ \hat{f} $ with respect to $ x_{i} $ exists and the limiting value is denoted by $  \frac{\partial \hat{f}(x)}{\partial x_{i}} $.
\end{defn}
\begin{defn} [S. Markov \cite{markov1979calculus}]
\label{Definition 1}
$ \hat{f}:\Omega \subseteq \mathbb{R}\rightarrow I(\mathbb{R}) $ is said to be $ \mu $-increasing in $ \Omega $ if $ \mu_{\hat{f}}(x) $ is increasing in $ \Omega $,  that is, $ \mu_{\hat{f}}(x_{1})\leq \mu_{\hat{f}}(x_{2}) $  for $x_{1}, x_{2}\in \Omega $, satisfying $ x_{1}< x_{2} $,  otherwise $ \hat{f} $ is called $ \mu $-decreasing in $ \Omega $. $ \hat{f}$ is said to be monotonic in $\Omega$ if it is either $ \mu $-increasing or $ \mu $-decreasing in $ \Omega $.
\end{defn}
$ \ominus_{gH}  \hat{f} (x)= \hat{0}\ominus_{gH} \hat{f}(x)= (-1) \hat{f}(x) $. So $ \mu_{(\ominus_{gH}  \hat{f})}(x)= \mu_{\hat{f}}(x) $.  Hence,   $ \hat{f}:\mathbb{R}\rightarrow I(\mathbb{R}) $  is $ \mu $- increasing implies $ \ominus_{gH}\hat{f} $ is also $ \mu $-increasing.
 An interval valued function may be neither $ \mu $-increasing nor $ \mu $-decreasing in $\mathbb{R}$. (Example $ \hat{f}(x)=[1,3]x^{2}, x\in \mathbb{R}$).
\section{Calculus of $ \hat{f}: \mathbb{R}^n \rightarrow I(\mathbb{R}) $ using $ \mu $ monotonic property}
 $ \mu $-monotonic property of an interval valued function plays an important role while developing calculus of interval valued function in higher dimension. In the light of $ \mu $-monotonic property of interval valued function in single variable in \cite{markov1979calculus}, we first focus on  $ \mu $-monotonicity  in higher dimension.\\ \\
Consider $ \hat{f}:\mathbb{R}^n\rightarrow I(\mathbb{R}) $, $ \hat{f}(x)=[\underline{f}(x),\overline{f}(x)] $, $\underline{f},\overline{f}: \mathbb{R}^n\rightarrow \mathbb{R}$. Denote $ \Lambda_{n} \triangleq \left\lbrace 1,2,\cdots,n\right\rbrace $ and $(x:ih_{i})\triangleq (x_1,x_2,...,x_i+h_{i},...x_n)$. $\Omega\subseteq\mathbb{R}^{n}$.
\begin{defn}[Component-wise $ \mu   $-monotonic property] \label{mui}
$\hat{f}:\mathbb{R}^n\rightarrow I(\mathbb{R})$ is said to be
\begin{itemize}
 \item $\mu-$increasing in $\Omega$ with respect to $i^{th}$ component if
$ \mu_{\hat{f}}(x)\leq \mu_{\hat{f}}(x:ih_{i})$ whenever $x_{i}< x_{i}+h_{i}$,\\ $\forall x, (x:ih_{i})\in \Omega
$,
\item  $\mu-$decreasing in $\Omega$ with respect to $i^{th}$ component if
$\mu_{\hat{f}}(x))\geq \mu_{\hat{f}}(x:ih_{i})$ whenever $x_{i}< x_{i}+h_i$, $\forall x ,(x:ih_{i})\in \Omega$,
\item  $ \mu $-monotonic with respect to $ x_{i} $ if it is either $ \mu $-increasing or $ \mu $-decreasing with respect to $i^{th}$ component,
 \item  strictly $\mu-$increasing  (decreasing) in $\Omega$ with respect to $i^{th}$ component if
$\mu_{\hat{f}}(x) <(>) \mu_{\hat{f}}(x:ih_{i})$ whenever $x_{i}< x_{i}+h_{i},~\forall x, (x:ih_{i})\in \Omega
$, (In a similar way other strictly $ \mu $ monotonic properties can be defined.)
\item  non $ \mu $-monotonic with respect to $i^{th}$ component if  either $ \hat{f} $ is $ \mu $-increasing  with respect to $ i^{th} $ component when $ x_i+h_i<x_i $, $ h_i<0 $ and $ \mu $-decreasing with respect to $ i^{th} $ component when $ x_i<x_i+h_i $, $ h_i>0 $, or  $ \mu $-decreasing  with respect to $ i^{th} $ component when $ x_i+h_i<x_i $, $ h_i<0 $ and $ \mu $-increasing  with respect to $ i^{th} $ component when $ x_i<x_i+h_i $, $ h_i>0 $.
    \end{itemize}
\end{defn}
Using this definition it is easy to show that if $ \mu_{\hat{f}}$ is differentiable at $x$ (that is $\underline{f}$ and $\overline{f}$ are differentiable at $x$), then  $\hat{f}$ is $ \mu $-increasing or $ \mu $-decreasing at $x$ with respect to $i^{th}$ component if $ \frac{\partial \mu_{\hat{f}}(x)}{\partial x_{i}}\geq 0$ or $ \frac{\partial \mu_{\hat{f}}(x)}{\partial x_{i}}\leq 0$ respectively.
\begin{note} From Definition \ref{Partial derivatives}, one may note that existence of partial derivative of an interval valued function at a point may not guarantee  the existence of partial derivatives of the lower and upper bound functions at that point. \\
 Consider $ \hat{f}(x_{1},x_{2})=\hat{a}x_{1} \oplus \hat{b}x_{2}^{2}$ for $ \hat{a},\hat{b} \in I(\mathbb{R})$, $where$ $\mu(\hat{a})>0$.
 Therefore $ \underline{f}(x_{1},x_{2})=
	\begin{cases}
	\underline{a}x_{1}+ \underline{b}x_{2} ^2 &\text{if $ x_1 \geq 0 $}\\
	\overline{a}x_{1}+ \underline{b}x_{2} ^2 &\text{if $ x_1 < 0 $}
	\end{cases}
	$.\\   $ \overline{f}(x_{1},x_{2})=
	\begin{cases}
	\overline{a}x_{1}+ \overline{b}x_{2} ^2 &\text{if $ x_1 \geq 0 $}\\
	\underline{a}x_{1}+ \overline{b}x_{2} ^2 &\text{if $ x_1 < 0 $}
	\end{cases}
	$.\\
	One can easily check that
	$ \frac{\partial \hat{f}(0,0)}{\partial x_{1}}=\hat{a} $, where as $ \frac{\partial \underline{f}(0,0)}{\partial x_{1}} $ and $ \frac{\partial \overline{f}(0,0)}{\partial x_{1}} $ do not exist.\end{note}
\begin{theorem}\label{th}
Let $\Omega$ be an open set in $\mathbb{R}^{n}$, $ \hat{f}:\Omega\rightarrow I(\mathbb{R}) $ be  $ \hat{f}(x)=[\underline{f}(x),\overline{f}(x)] $.
\begin{enumerate}
\item If  $\frac{\partial \underline{f}(x)}{\partial x_{i}}$ and $\frac{\partial \overline{f}(x)}{\partial x_{i}}$ exist, then $ \frac{\partial \hat{f}(x)}{\partial x_{i}}$ exists and  $ \frac{\partial \hat{f}(x)}{\partial x_{i}}=\left[ \frac{\partial \underline{f}(x)}{\partial x_{i}} \vee \frac{\partial \overline{f}(x)}{\partial x_{i}}\right]$.
\item  Suppose $\frac{\partial \hat{f}(x)}{\partial x_{i}}$ exists.\\

     $ (a) $ If $ \hat{f} $  is  non $ \mu -$monotonic  with respect to $i^{th}$ component  in $ nbd(x) $ and if the lateral partial derivatives of $ \underline{f} $ and $ \overline{f} $ respectively with respect to $ x_i $ i.e. $ \left( \frac{\partial \underline{f}(x)}{\partial x_i} \right)_{-} $, $\left( \frac{\partial \underline{f}(x)}{\partial x_i} \right)_{+}  $ and $ \left( \frac{\partial \overline{f}(x)}{\partial x_i} \right)_{-} $, $ \left( \frac{\partial \overline{f}(x)}{\partial x_i} \right)_{+} $ exist , then
$     \left( \frac{\partial \overline{f}(x)}{\partial x_i} \right)_{-}=\left( \frac{\partial \overline{f}(x)}{\partial x_i} \right)_{+} $; $\left( \frac{\partial \underline{f}(x)}{\partial x_i} \right)_{+}= \left( \frac{\partial \overline{f}(x)}{\partial x_i} \right)_{-} $ hold and
     \begin{align*}
	\frac{\partial \hat{f}(x)}{\partial x_i} &= \left[ \min \left \lbrace \left( \frac{\partial \underline{f}(x)}{\partial x_i} \right)_{-}, \left( \frac{\partial \overline{f}(x)}{\partial x_i} \right)_{-}  \right \rbrace, \max \left \lbrace \left( \frac{\partial \underline{f}(x)}{\partial x_i} \right)_{-}, \left( \frac{\partial \overline{f}(x)}{\partial x_i} \right)_{-}  \right \rbrace \right]\\
	&= \left[ \min \left \lbrace \left( \frac{\partial \underline{f}(x)}{\partial x_i} \right)_{+}, \left( \frac{\partial \overline{f}(x)}{\partial x_i} \right)_{+}  \right \rbrace, \max \left \lbrace \left( \frac{\partial \underline{f}(x)}{\partial x_i} \right)_{+}, \left( \frac{\partial \overline{f}(x)}{\partial x_i} \right)_{+}  \right \rbrace \right]
	\end{align*}
   \\
 $ (b) $ If $ \hat{f} $  is  $ \mu -$monotonic  with respect to $i^{th}$ component  in $ nbd(x) $
then\\  $\frac{\partial \hat{f}(x)}{\partial x_{i}}=\begin{cases}
     \left[\frac{\partial \underline{f}(x)}{\partial x_{i}},\frac{\partial \overline{f}(x)}{\partial x_{i}}\right] &\text{ if } \hat{f} \text{ is $\mu$-increasing }\\
     \left[\frac{\partial \overline{f}(x)}{\partial x_{i}},\frac{\partial \underline{f}(x)}{\partial x_{i}}\right] &\text{ if } \hat{f} \text{ is $\mu$-decreasing }
     \end{cases}$.
     \end{enumerate}
\end{theorem}
\begin{proof}
\begin{enumerate}
\item
\begin{align*}
\frac{\partial \hat{f}(x)}{\partial x_{i}}
&= \lim_{h_{i} \rightarrow 0}\frac{1}{h_{i}}\left( \hat{f}(x:ih_{i})\ominus_{gH} \hat{f}(x)\right)\\
&= \lim_{h_{i} \rightarrow 0}\frac{1}{h_{i}}\left(  \left[ \underline{f} (x:ih_{i}),\overline{f}(x:ih_{i})\right] \ominus_{gH}
\left[ \underline{f}(x),\overline{f}(x)\right]\right)\\
 &=\left[ \left(  \lim_{h_i \rightarrow 0}\frac{\underline{f} (x:ih)- \underline{f}(x)}{h_i}\right) \vee
\left(\lim_{h_i \rightarrow 0} \frac{\overline{f}(x:ih) -\overline{f}(x)}{h_i}\right) \right]
\\
&= \left[ \frac{\partial \underline{f}(x)}{\partial x_{i}}\vee \frac{\partial \overline{f}(x)}{\partial x_{i}}\right] \text{    ( Since $\frac{\partial \underline{f}(x)}{\partial x_{i}}  $ and $  \frac{\partial \overline{f}(x)}{\partial x_{i}} $ exist.)}
\end{align*}
Hence $ \frac{\partial \hat{f}(x)}{\partial x_{i}} $ exists.
\item
$(a)$
Suppose $ \hat{f} $ is non $ \mu $-monotonic with respect to  $ i^{th} $ component.
Then
$ \hat{f} $ is
 $ \mu $-increasing\\($\mu $-decreasing)   when $ x_i+h_i<x_i $, $ h_i<0 $ and $ \mu $-decreasing ($ \mu $-increasing)  when $ x_i<x_i+h_i $, $ h_i>0 $.
 \\ That is,
$\mu_{\hat{f}}(x:ih_i) \leq ~ ( \geq)~ \mu_{\hat{f}}(x)$  whenever $ x_i+h_i<x_i $, $ h_i<0 $ and \\$
\mu_{\hat{f}}(x:ih_i) \leq ~( \geq)~ \mu_{\hat{f}}(x)$  whenever $ x_i <x_i+h_i $, $ h_i>0 $.
 Hence
$$ \overline{f}(x:ih_i)- \overline{f}(x) \leq ~ ( \geq)~  \underline{f}(x:ih_i)- \underline{f}(x) \mbox{~~whenever~~} x_i+h_i<x_i , ~h_i<0 $$
 and
  $$ \overline{f}(x:ih_i)- \overline{f}(x) \leq ~ ( \geq)~ \underline{f}(x:ih_i)- \underline{f}(x)\mbox{~~whenever~~}x_i <x_i+h_i,~ h_i>0 .$$
  Therefore
\begin{align}
 \frac{  \underline{f}(x:ih_i)- \underline{f}(x)}{h_i} &\leq (\geq) \frac{\overline{f}(x:ih_i)- \overline{f}(x)}{h_i}  & \text{ whenever $ x_i+h_i<x_i $, $ h_i<0 $ } \label{in1} \\
 \frac{\overline{f}(x:ih_i)- \overline{f}(x)}{h_i} & \leq (\geq)   \frac{\underline{f}(x:ih_i)- \underline{f}(x)}{h_i} & \text{ whenever $ x_i <x_i+h_i $, $ h_i>0 $ } \label{in2}
\end{align}
Since $ \frac{\partial \hat{f}(x)}{\partial x_{i}} $ exists, so
\begin{equation}\label{cn1}
\frac{\partial \hat{f}(x)}{\partial x_{i}}=  \lim_{h_{i} \rightarrow 0^-} \frac{1}{h_{i}}\left( \hat{f}(x:ih_{i})\ominus_{gH} \hat{f}(x)\right)= \lim_{h_{i} \rightarrow 0^+} \frac{1}{h_{i}}\left( \hat{f}(x:ih_{i})\ominus_{gH} \hat{f}(x)\right)
\end{equation}
\begin{align*}
\frac{\partial \hat{f}(x)}{\partial x_{i}} &=\lim_{h_{i} \rightarrow 0^-} \frac{1}{h_{i}}\left( \hat{f}(x:ih_{i})\ominus_{gH} \hat{f}(x)\right)\\
&=\left[ \lim_{h_{i} \rightarrow 0^-} \frac{  \underline{f}(x:ih_i)- \underline{f}(x)}{h_i} \vee \lim_{h_{i} \rightarrow 0^-} \frac{  \overline{f}(x:ih_i)- \overline{f}(x)}{h_i}\right] \\
&=\left[ \left( \frac{\partial \underline{f}(x)}{\partial x_i} \right)_{-} \vee \left(\frac{\partial \overline{f}(x)}{\partial x_i} \right)_{-}\right]
\end{align*}
From \eqref{in1}
\begin{align}
\nonumber
&\frac{\partial \hat{f}(x)}{\partial x_{i}}= \\
& \begin{cases}
\left[ \left( \frac{\partial \underline{f}(x)}{\partial x_i} \right)_{-},\left(\frac{\partial \overline{f}(x)}{\partial x_i} \right)_{-}\right]  \text{ if $ \hat{f} $ is $ \mu $ increasing whenever  $x_{i}+h_i< x_i$, $h_i<0$}\\
\left[ \left( \frac{\partial \overline{f}(x)}{\partial x_i} \right)_{-},\left(\frac{\partial \underline{f}(x)}{\partial x_i} \right)_{-}\right]  \text{ if $ \hat{f} $ is $ \mu $ decreasing whenever  $x_{i}+h_i< x_i$, $h_i<0$}
\end{cases}
\label{in3}
\end{align}
Similarly from \eqref{in2}, it is easy to verify that
\begin{align}
\nonumber
&
\frac{\partial \hat{f}(x)}{\partial x_{i}}\\
\nonumber
&=\lim_{h_{i} \rightarrow 0^+} \frac{1}{h_{i}}\left( \hat{f}(x:ih_{i})\ominus_{gH} \hat{f}(x)\right)\\
&= \begin{cases}
\left[ \left( \frac{\partial \overline{f}(x)}{\partial x_i} \right)_{+},\left(\frac{\partial \underline{f}(x)}{\partial x_i} \right)_{+}\right]
  \text{ if $ \hat{f} $ is $ \mu $ decreasing whenever  $ x_i< x_{i}+h_i$, $h_i>0$}\\
\left[ \left( \frac{\partial \underline{f}(x)}{\partial x_i} \right)_{+},\left(\frac{\partial \overline{f}(x)}{\partial x_i} \right)_{+}\right] \text{ if $ \hat{f} $ is $ \mu $ increasing whenever $ x_i< x_{i}+h_i$, $h_i>0$}
\end{cases}
\label{in4}
\end{align}
Since $ \hat{f} $ is non $ \mu $ monotonic with respect to $ x_i $ and $\frac{\partial \hat{f}(x)}{\partial x_{i}}  $ exist, from \eqref{in3} and \eqref{in4},\\
$     \left( \frac{\partial \overline{f}(x)}{\partial x_i} \right)_{-}=\left( \frac{\partial \overline{f}(x)}{\partial x_i} \right)_{+} $ and $\left( \frac{\partial \underline{f}(x)}{\partial x_i} \right)_{+}= \left( \frac{\partial \overline{f}(x)}{\partial x_i} \right)_{-} $ hold. \\

Combining \eqref{cn1}, \eqref{in3} and \eqref{in4},
	\begin{align*}
	\frac{\partial \hat{f}(x)}{\partial x_i} &= \left[ \min \left \lbrace \left( \frac{\partial \underline{f}(x)}{\partial x_i} \right)_{-}, \left( \frac{\partial \overline{f}(x)}{\partial x_i} \right)_{-}  \right \rbrace, \max \left \lbrace \left( \frac{\partial \underline{f}(x)}{\partial x_i} \right)_{-}, \left( \frac{\partial \overline{f}(x)}{\partial x_i} \right)_{-}  \right \rbrace \right]\\
	&= \left[ \min \left \lbrace \left( \frac{\partial \underline{f}(x)}{\partial x_i} \right)_{+}, \left( \frac{\partial \overline{f}(x)}{\partial x_i} \right)_{+}  \right \rbrace, \max \left \lbrace \left( \frac{\partial \underline{f}(x)}{\partial x_i} \right)_{+}, \left( \frac{\partial \overline{f}(x)}{\partial x_i} \right)_{+}  \right \rbrace \right]
	\end{align*}
	
	 \textbf{(b)} If $ \hat{f} $ is $ \mu $-increasing with respect to $ i^{th}$ component  in $nbd(x)$ then\\
$ \mu_{\hat{f}}(x)\leq \mu_{\hat{f}}(x:ih_{i}) ~~ \text{ for } ~ x_{i}<x_{i}+h_{i} $. Hence
\begin{equation*}
 \frac{\underline{f}(x:ih_{i})-\underline{f}(x)}{h_{i}} \leq \frac{\overline{f}(x:ih_{i})-\overline{f}(x)}{h_{i}}
\end{equation*}
Similarly, if $ \hat{f} $ is $ \mu $-decreasing with respect to $ i^{th}$ component  then
 $$ \frac{\overline{f}(x:ih_{i})-\overline{f}(x)}{h_{i}} \leq  \frac{\underline{f}(x:ih)-\underline{f}(x)}{h_i} $$ Since $\hat{f}$ is $\mu$ monotonic and $\frac{\partial \hat{f}(x)}{\partial x_{i}}$ exists, so
\begin{align*}
&\frac{\partial \hat{f}(x)}{\partial x_{i}}\\
&=\lim_{h_{i} \rightarrow 0}\frac{1}{h_{i}}\left( \hat{f}(x:ih_{i})\ominus_{gH} \hat{f}(x)\right)\\
&=\begin{cases}
\left[ \left( \lim_{h_{i} \rightarrow 0}\frac{\underline{f} (x:ih_{i})- \underline{f}(x)}{h_{i}}\right) , \left(\lim_{h \rightarrow 0} \frac{\overline{f}(x:ih_{i}) - \overline{f}(x)}{h_{i}}\right) \right] & \text{ if }  \hat{f} \text{is $\mu-$ increasing in $nbd(x)$ }\\
\left[\left(\lim_{h_{i} \rightarrow 0} \frac{\overline{f}(x:ih_{i}) - \overline{f}(x)}{h_{i}}\right) , \left(\lim_{h_{i} \rightarrow 0} \frac{\underline{f} (x:ih_{i})- \underline{f}(x)}{h_{i}}\right) \right] & \text{ if }  \hat{f} \text{is $\mu-$ decreasing in $nbd(x)$}
  \end{cases}\\
&=\begin{cases}
\left[ \frac{\partial \underline{f}(x)}{\partial x_{i}}, \frac{\partial \overline{f}(x)}{\partial x_{i}}\right] &\text{if } \hat{f} \text{is $\mu-$ increasing in $\Omega$} \\
\left[  \frac{\partial \overline{f}(x)}{\partial x_{i}},\frac{\partial \underline{f}(x)}{\partial x_{i}}\right] &\text{if } \hat{f} \text{is $\mu-$ decreasing in $\Omega$ }
\end{cases}
\end{align*}
	\end{enumerate}
	Hence the theorem follows.
\end{proof}
Gradient of interval valued function at a point $ x \in \mathbb{R}^n $ is an interval vector and is denoted by $$\nabla \hat{f}(x)\triangleq(\frac{\partial \hat{f}(x)}{\partial x_{1}},\frac{\partial \hat{f}(x)}{\partial x_{2}},\cdots,\frac{\partial \hat{f}(x)}{\partial x_{n}})^{T}$$
Following result is from Ref. \cite{osuna2015optimality}, the $gH$- differentiability of $ \hat{f} $ over $ \mathbb{R}^n $.
\begin{defn}[R~Osuna-G{\'o}mez et. al\cite{osuna2015optimality}]
If all the partial derivatives of $ \hat{f}:\mathbb{R}^{n}\rightarrow I(\mathbb{R}) $ exist and continuous in the neighbourhood of  $ x\in \mathbb{R}^{n} $, then $ \hat{f} $ is $ gH$-differentiable at $x$.
\end{defn}
 Following this definition, in the light of calculus of real valued function of several variables, the  gH differentiability of $ \hat{f}: \mathbb{R}^n \rightarrow I(\mathbb{R}) $ may be restated in terms of interval valued error function. For a gH differentiable function  $\hat{f}: \mathbb{R}^n \rightarrow I(\mathbb{R}) $, partial derivatives of $\hat{f}$ exists and there exists an interval valued error function \\ $ \hat{E}_{x}: \mathbb{R}^n\rightarrow I(\mathbb{R}) $, satisfying $\displaystyle \lim_{\|h\|\rightarrow 0} \hat{E}_{x}(h)= \hat{0}$ such that \\$ \hat{w}(\hat{f}(x_0);h) \ominus_{gH} \sum\limits_{i=1}^{n}\left(  h_{i}\frac{\partial \hat{f}(x)}{\partial x_{i} }\right)=(\|h\|  \hat{E}_{x}(h)) $ hold. Using $ gH $-difference \eqref{dif},  this concept can be stated in following form.\\
  An interval valued function $ \hat{f}:\mathbb{R}^{n}\rightarrow I(\mathbb{R}) $  is  gH differentiable at $x\in \mathbb{R}^n$ if $ \nabla\hat{f}(x) \in (I(\mathbb{R}))^{n}  $ exists and there exists an interval valued error function $ \hat{E}_{x}(h) \in I(\mathbb{R}) $, $h\in \mathbb{R}^n$ such that
\begin{align}
 \hat{w}(\hat{f}(x);h)=\sum\limits_{i=1}^{n}\left(  h_{i}\frac{\partial \hat{f}(x)}{\partial x_{i} }\right) \oplus (\|h\|  \hat{E}_{x}(h))   \label{ndef1}\\
 \nonumber \mbox{ or ~~~~~~~~~~~~~~~~~~~~~~~~~~~}\\
 \hat{w}(\hat{f}(x);h)\oplus (-1)(\|h\|  \hat{E}_{x}(h)) = \sum\limits_{i=1}^{n}\left(  h_{i} \frac{\partial \hat{f}(x)}{\partial x_{i} }\right) \label{ndef2}
 \end{align}
hold for $ \|h\| < \delta $  for some $ \delta>0 $ with $\displaystyle \lim_{\|h\|\rightarrow 0} \hat{E}_{x}(h)= \hat{0}$. This form will be useful to study the differentiability of composite interval valued function in next theorem.
\begin{theorem}
\label{chainrule} 
Suppose $ \hat{f}\colon \mathbb{R}^n \to I(\mathbb{R})$, denoted by $ \hat{f}(x) \triangleq \hat{f}(x_1, x_2, \cdots, x_n) $, is an interval valued gH differentiable function at $ x_0$ and $u \colon \mathbb{R}^{m} \to \mathbb{R}^{n} $ denoted by  $ u(t) \triangleq \begin{pmatrix}
u_{1}(t_1, t_2, \cdots, t_m) &u_{2}(t_1, t_2, \cdots, t_m)& \cdots & u_{n}(t_1, t_2, \cdots, t_m)
\end{pmatrix}^T $  is differentiable at `$ a $' with Jacobian matrix  $ Du(a)$ of order $ n \times m $. If $ x_0= u(a) $ then the composite function $ \hat{G}\triangleq \hat{f}\circ u : \mathbb{R}^m \rightarrow I(\mathbb{R})  $ is gH differentiable at $ a $, and $ \nabla \hat{G}(a)=  Du(a)^{T} \nabla \hat{f}(x_0)$.
\end{theorem}
\begin{proof}  Since $ x_0= u(a) $, the composition function $ \hat{\Phi}:= \hat{f}\circ u : \mathbb{R}^m \rightarrow I(\mathbb{R})  $ is defined in the neighbourhood of $ a $. For sufficiently small $ \Vert h \Vert $,
\begin{align}
\nonumber
\hat{w}\left( \hat{G}(a);h  \right)=\hat{G}(a+h) \ominus_{gH} \hat{G}(a)&= \hat{f}\left( u(a+h)\right) \ominus_{gH} \hat{f}\left( u(a)\right)\\ \nonumber
&=\hat{f}(x_0+v)\ominus_{gH} \hat{f}(x_0),  \text{ where } v=u(a+h)-x_0 \\
&= \hat{w}(\hat{f}(x_0);v) \label{cm4}
\end{align}
Since $ \hat{f} $ is $ gH $-differentiable, from \eqref{ndef1} and \eqref{ndef2} there exists an error function $ \hat{E}_{x_0}(h) $ such that
\begin{align}
\hat{w}(\hat{f}(x_0);v)= \left(\sum_{i=1}^{n} v_i \frac{\partial \hat{f}(x_0) }{\partial x_i}\right)\oplus \Vert v \Vert  \hat{E}_{x_0}(v) \label{nde1}\\
\nonumber \mbox{ or ~~~~~~~~~~~~~~~~~~~~~}\\
\hat{w}(\hat{f}(x_0);v) \oplus (-1)\Vert v \Vert  \hat{E}_{x_0}(v)= \left(\sum_{i=1}^{n} v_i \frac{\partial \hat{f}(x_0) }{\partial x_i}\right)\label{nde2}
\end{align}
hold for $ \Vert v \Vert < \delta^\prime $  with $ \delta^\prime > 0$  where $\underset{\Vert v \Vert \rightarrow 0}\lim \hat{E}_{x_0}(v)=\hat{0}$.
Using Taylor's expansion for  $ u $ at $ a $,
\begin{align}
v= u(a+h)-x_0= Du(a)h \mbox{ }+ \Vert h \Vert E_{a} (h) \label{cm1}\\
\nonumber
\mbox{ for~$\Vert h \Vert < \delta$~with~ $\delta>0$ ~where~$\underset{\Vert h \Vert \rightarrow 0}\lim E_{a}(h)=0$}.
\end{align}
 Ignoring the error term,  $ v \approx Du(a)h$.
From  \eqref{nde1} and \eqref{nde2},
 \begin{align}
\hat{w}(\hat{f}(x_0);v)  \approx (Du(a)h)^T  \nabla\hat{f}(x_0)  \oplus \Vert h \Vert \frac{\Vert v \Vert}{\Vert h \Vert}   \hat{E}_{x_0}(v) \label{cm5}\\
\nonumber \mbox{ or ~~~~~~~~~~~~~~~~~~~~~}\\
\hat{w}(\hat{f}(x_0);v) \oplus (-1)\Vert h \Vert \frac{\Vert v \Vert}{\Vert h \Vert}   \hat{E}_{x_0}(v)\approx (Du(a)h)^T  \nabla\hat{f}(x_0) \label{cm6}
 \end{align}
 hold. From \eqref{cm1},  $ \Vert h \Vert \rightarrow 0 $ implies $\Vert v\Vert \rightarrow 0 $.
  $  \frac{\Vert v \Vert}{\Vert h\Vert}  $  remains bounded as $ \Vert h \Vert \rightarrow 0  $ since \begin{align*}
 \Vert v \Vert & \leq \Vert  Du(a)h \Vert + \Vert h \Vert   \Vert E_{a} (h) \Vert \\
   & \leq  \Vert h \Vert  \left( M +  \Vert E_{a} (h) \Vert\right)  \text{ where } M\triangleq \sum_{i=1}^n \nabla u_i(a).
 \end{align*}
 For $ \displaystyle h \in \mathbb{R}^m,~Du(a)h = \left( \sum_{j=1}^m h_j \frac{\partial u_{i}(a)}{\partial t_j}\right) _{n \times 1}~\forall~  ~i=1,2, \cdots, n $. Therefore\\ $ \displaystyle v_i \approx \sum_{j=1}^{m} h_j \frac{\partial u_{i}(a)}{\partial t_j} ~\forall~ i=1,2, \cdots, n $.\\
 From \eqref{cm4}, \eqref{cm5} and \eqref{cm6},
 \begin{align*}
 \hat{w}\left( \hat{G}(a);h  \right)  = \left(  \sum_{i=1}^{n} \left( \sum_{j=1}^m h_j \frac{\partial u_{i}(a)}{\partial t_j}\right) \frac{\partial \hat{f}(x_0)}{\partial x_i} \right)    \oplus \Vert h \Vert  \hat{E}(h)\\
 \mbox{ or~~~~~~~~~~~~~~~~~~~~~~~~~~~~~~}\\
\hat{w}\left( \hat{G}(a);h  \right) \oplus (-1)\Vert h \Vert \hat{E}(h)= \left(  \sum_{i=1}^{n} \left( \sum_{j=1}^m h_j \frac{\partial u_{i}(a)}{\partial t_j}\right) \frac{\partial \hat{f}(x_0)}{\partial x_i} \right)
 \end{align*}
   hold  where $ \hat{E}(h)= \frac{\Vert v \Vert}{\Vert h \Vert}  \hat{E}_{x_0}(v) $ and $\hat{E}(h) \rightarrow \hat{0}$ as $ \Vert h \Vert \rightarrow 0 $.
 Hence $ \hat{G} $ is gH differentiable at $ a $ and from the above expression,  $ \nabla \hat{G}(a)=Du(a)^{T}  \nabla\hat{f}(x_0)             $.
\end{proof}
 \begin{corollary}\label{comp}
In particular for $ m=1 $ ( i.e. for $ u: \mathbb{R} \rightarrow \mathbb{R}^n $), the composite function  $ \hat{g}\triangleq \hat{f}\circ u : \mathbb{R} \rightarrow I(\mathbb{R})  $ is gH differentiable at $ a $, and $  \hat{g}^{\prime}(a)= Du(a)^{T}  \nabla \hat{f}(x_0)= \sum_{i=1}^{n} u_{i}^{\prime}(a) \frac{\partial \hat{f}(x_0)}{\partial x_i} $ where $ Du(a)= \begin{pmatrix}
u_{1}^{\prime}(a) & u_{2}^{\prime}(a)& \cdots & u_{n}^{\prime}(a)
\end{pmatrix}^T $.
\end{corollary}
Proof of this result is straight forward from the above theorem.
\begin{note}
From Corollary \ref{comp}, one may observe that the expression for $ \underline{g}^{\prime} (a) $ and $ \overline{g}^{\prime} (a) $ may not coincide with either the expression $ Du(a)^{T} \nabla \underline{f}(x_0) $ or $ Du(a)^{T} \nabla \overline{f}(x_0) $ in general. Under certain restrictions this condition may hold, which is discussed below.
\begin{enumerate}
\item If $ \hat{f} $ is $ \mu $ increasing (decreasing) with respect to $ x_i $ at $ x_0 $ $ \forall~i $ and  $ u_i $ is monotonically increasing (decreasing) at $ a $ $ \forall~i $, then $ \underline{g}^{\prime} (a)= Du(a)^{T} \nabla \underline{f}(x_0)  $ and $\overline{g}^{\prime} (a)= Du(a)^{T} \nabla \overline{f}(x_0)  $.
\item If $ \hat{f} $ is $ \mu $ decreasing (increasing) with respect to $ x_i $ at $ x_0 $ $ \forall~i $ and $ u_i $ is monotonically increasing (decreasing) at $ a $ $ \forall~i $, then $ \underline{g}^{\prime} (a)= Du(a)^{T} \nabla \overline{f}(x_0)  $ and $\overline{g}^{\prime} (a)= Du(a)^{T} \nabla \underline{f}(x_0)  $.
\end{enumerate}
\end{note}

The basic idea in Theorem \ref{th} can be extended to study the higher order partial derivative of interval valued function.
\begin{prop}
\begin{enumerate}
	\item  If the  partial derivatives of $ \frac{\partial \underline{f}(x)}{\partial x_i} $ and $ \frac{\partial \overline{f}(x)}{\partial x_i} $ exist with respect to $ x_j $, then the  partial derivative of $ \frac{\partial \hat{f}(x)}{\partial x_i} $ also exists with respect to $ x_j $ and
	\begin{align*}
	\frac{\partial^{2} \hat{f}(x)}{\partial x_{j}\partial x_{i}}= \left [ \frac{\partial^{2} \underline{f}(x)}{\partial x_{j}\partial x_{i}} \vee   \frac{\partial^{2} {\overline{f}(x)}}{\partial x_{j}\partial x_{i}} \right ]  .
	\end{align*}
	\item If $ \frac{\partial^{2} \hat{f}(x)}{\partial x_{j}\partial x_{i}} $ exists and $ \frac{\partial \hat{f}(x)}{\partial x_i} $ is $ \mu $ monotonic with respect to $ x_j $, then the  partial derivatives of $ \frac{\partial \underline{f}(x)}{\partial x_i} $ and $ \frac{\partial \overline{f}(x)}{\partial x_i} $  also exist with respect to $ x_j $ and \\ $ \frac{\partial^{2} \hat{f}(x)}{\partial x_{j}\partial x_{i}}= \begin{cases}
	\left [ \frac{\partial^{2} \underline{f}(x)}{\partial x_{j}\partial x_{i}} ,  \frac{\partial^{2} {\overline{f}(x)}}{\partial x_{j}\partial x_{i}} \right ]  \text{ if $\frac{\partial \hat{f}(x)}{\partial x_i} $ is $ \mu $ increasing with respect to $ x_j $  } \\
	\left [ \frac{\partial^{2} \overline{f}(x)}{\partial x_{j}\partial x_{i}} ,  \frac{\partial^{2} {\underline{f}(x)}}{\partial x_{j}\partial x_{i}} \right ]  \text{ if $\frac{\partial \hat{f}(x)}{\partial x_i} $ is $ \mu $ decreasing with respect to $ x_j $  }
		\end{cases}
		$
	
\end{enumerate}
\end{prop}

\begin{proof}
	The proof of this proposition directly follows from proof of Theorem \ref{th}.
\end{proof}
 The Hessian of $ \hat{f}(x) $ is an $ n\times n $ interval matrix denoted by  $ \nabla^{2}\hat{f}(x)$ whose $ (ij)^{th} $ component is an interval $ \frac{\partial^{2}\hat{f}(x)}{\partial x_{i} \partial x_{j}}$.
 \begin{ex}
 $\hat{f}(x_{1},x_{2})=[1,2]x_{1}^{3}e^{[1,2]x_{2}}$. Then \\
 $\hat{f}(x_{1},x_{2})=\begin{cases}
[x_{1}^{3}e^{x_{2}},2x_{1}^{3}e^{2x_{2}}] &\text{ when } x_{1}\geq 0,x_{2}\geq 0 \\
[2x_{1}^{3}e^{2x_{2}}, x_{1}^{3}e^{x_{2}}] &\text{ when } x_{1}\leq 0, x_{2}\geq 0 \\
[2x_{1}^{3}e^{x_{2}}, x_{1}^{3}e^{2x_{2}}] &\text{ when } x_{1}\leq 0, x_{2}\leq 0 \\
[x_{1}^{3}e^{2x_{2}}, 2x_{1}^{3}e^{x_{2}}] &\text{ when } x_{1}\geq 0, x_{2}\leq 0 \\
\end{cases}$.\\
Consider $\hat{f}(x_{1},x_{2})=[2x_{1}^{3}e^{x_{2}},x_{1}^{3}e^{2x_{2}}]$ for $ x_{1}\leq 0,x_{2}\leq 0 $. $ \mu_{\hat{f}}(x_1, x_2)= x_{1}^{3}e^{2x_{2}}- 2x_{1}^{3}e^{x_{2}}  $. $ \hat{f} $ is $ \mu $ decreasing with respect to $ x_1 $ and $ \mu $ increasing with respect to $ x_2 $. \\
Therefore $\frac{\partial \hat{f}}{\partial x_{1}}=[3x_{1}^{2}e^{2x_{2}},6x_{1}^{2}e^{x_{2}}]$, $\frac{\partial \hat{f}}{\partial x_{2}}=[2x_{1}^{3}e^{x_{2}},2x_{1}^{3}e^{2x_{2}}]$.\\
$\frac{\partial \hat{f}}{\partial x_{1}}$  and $\frac{\partial \hat{f}}{\partial x_{2}}$, both are $ \mu $ decreasing with respect to $ x_1 $ and $ \mu $ increasing with respect to $ x_2 $. Therefore $\frac{\partial^{2} \hat{f}}{\partial x_{1}^{2}}=[12x_{1}e^{x_{2}},6x_{1}e^{2x_{2}}]$, $\frac{\partial^{2} \hat{f}}{\partial x_{2}^{2}}=[2x_{1}^{3}e^{x_{2}},4x_{1}^{3}e^{2x_{2}}]$, $\frac{\partial^{2} \hat{f}}{\partial x_{1}x_{2}}=[6x_{1}^2e^{2x_{2}},6x_{1}^2e^{x_{2}}]=\frac{\partial^{2} \hat{f}}{\partial x_{2}x_{1}} $.\\
Hence Hessian of $\hat{f}$ at $(-1,-1)$ becomes
 $\nabla^{2}\hat{f}(-1,-1)=\begin{pmatrix}
[-12e^{-1},-6e^{-2}]& [6e^{-2},6e^{-1}]\\
[6e^{-2},6e^{-1}] & [-2e^{-1},-4e^{-2}]
\end{pmatrix}
$.
\end{ex}

\section{Expansion of interval valued function}
\subsection{Expansion of interval valued function over $\mathbb{R}$}
From Definition \ref{Definition 2}, one may conclude that  $ \hat{f} $ is n times gH differentiable at $ x $ if $\underset{h \rightarrow 0} \lim \frac{\hat{f}^{(n-1)}(x+h) \ominus_{gH}\hat{f}^{(n-1)}(x) }{h} $ exists. The limiting value is called the $n^{th}$ order gH derivative of $ \hat{f} $ at $ x $ and denoted by $\hat{f}^{n}(x)  $.

\begin{prop}\label{lma}
Suppose $ g:\mathbb{R}\rightarrow \mathbb{R}  $ is a real valued differentiable function  and  $ \hat{f}:\mathbb{R}\rightarrow I(\mathbb{R}) $ be first order gH differentiable and $ \mu $ monotonic function. Then $( g\hat{f})$ is gH differentiable and $( g\hat{f})^{'} (x)= [(g \underline{f})^{\prime}(x) \vee (g \overline{f})^{\prime}(x)] $.
\end{prop}
\begin{proof}
$$( g\hat{f})^{\prime}(x)=\lim_{h\rightarrow 0}\frac{(g\hat{f})(x+h)\ominus_{gH} (g\hat{f})(x)}{h}$$
$$=\lim_{h\rightarrow 0}\frac{g(x+h)[\underline{f}(x+h),\overline{f}(x+h)]\ominus_{gH}g(x)[\underline{f}(x),\overline{f}(x)]}{h}$$
Since $ g $ is differentiable, for sufficiently small $ h $, $ g(x+h) $ and $ g(x) $ are of same sign.
$$( g\hat{f})^{\prime}(x) =\lim_{h\rightarrow 0}\left[ \frac{g(x+h)\underline{f}(x+h)-g(x)\underline{f}(x)}{h} \vee   \frac{g(x+h)\overline{f}(x+h)-g(x)\overline{f}(x)}{h}\right] $$
Since $\hat{f}$  is  gH differentiable and $\mu$-monotonic and $ g $ is differentiable, $ g\underline{f} $ and $ g \overline{f} $ are differentiable. Hence
\begin{align*}
&\lim_{h\rightarrow 0}\frac{g(x+h)\underline{f}(x+h)-g(x)\underline{f}(x)}{h} \\
=&\lim_{h\rightarrow 0}\frac{g(x+h)-g(x)}{h}\underline{f}(x)+\lim_{h\rightarrow 0}g(x+h)\frac{\underline{f}(x+h)-\underline{f}(x)}{h} \\
=&g^{\prime}(x)\underline{f}(x)+g(x)\underline{f}^{\prime}(x)=(g \underline{f})^{\prime}(x)
\end{align*}
In a similar way, it is easy to verify that $\lim_{h\rightarrow 0} \frac{g(x+h)\overline{f}(x+h)-g(x)\overline{f}(x)}{h}=(g \overline{f})^{\prime}(x)$.\\
Hence $
( g\hat{f})^{\prime}(x) = [(g \underline{f})^{\prime}(x) \vee (g \overline{f})^{\prime}(x)]$.
\end{proof}

Chalco et. al. \cite{chalco2011generalized} justified that the concept of $ gH $ difference is same as Markov difference($\ominus_{M}$), introduced by S. Markov \cite{markov1977extended} in case of compact set of intervals.  Therefore Theorem 7 and Theorem 9 from Ref. \cite{markov1979calculus} can be restated in terms of gH difference as in Theorem \ref{Theorem 7} and Theorem \ref{Mean-Value theorem for interval functions} below.  Theorem \ref{Theorem 7} is same as Theorem 7 of Ref. \cite{markov1979calculus} and Theorem \ref{Mean-Value theorem for interval functions} is same as Theorem 9 of Ref. \cite{markov1979calculus}, obtained by replacing Markov difference with gH difference. Hence the proofs are omitted. These two results are used for the theoretical developments in future part of this article.
\begin{theorem}[S. Markov \cite{markov1979calculus}]
\label{Theorem 7}
Suppose $ \hat{f},\hat{g}:\Omega\subseteq \mathbb{R}\rightarrow I(\mathbb{R}) $ are $ \mu $-monotonic and gH differentiable in $ \Omega $.
\begin{enumerate}[(i)]
\item  If $ \hat{f} $ and $ \hat{g} $ are equally $ \mu $-monotonic (both are $ \mu $-increasing or $ \mu $-decreasing) then $ (\hat{f}\oplus\hat{g})^{\prime}=\hat{f}^{\prime}\oplus \hat{g}^{\prime} $ and $(\hat{f}\ominus_{gH}\hat{g})^{\prime}=\hat{f}^{\prime}\ominus_{gH} \hat{g}^{\prime}   $;
\item  If $ \hat{f} $ and $ \hat{g} $ are differently $ \mu $-monotonic(one is  $ \mu $-increasing and the other is $ \mu $-decreasing) then $ (\hat{f}\oplus\hat{g})^{\prime}=\hat{f}^{\prime}\ominus_{gH}(\ominus_{gH}\hat{g}^{\prime}) $ and $(\hat{f}\ominus_{gH}\hat{g})^{\prime}=\hat{f}^{\prime}\oplus(\ominus_{gH} \hat{g}^{\prime}) $.
\end{enumerate}
\end{theorem}
\begin{theorem}[S. Markov \cite{markov1979calculus}]
\label{Mean-Value theorem for interval functions}
If $ \hat{f}:\mathbb{R}\rightarrow I( \mathbb{R})$ is continuous in $ \Delta$, where $ \Delta=\left[\alpha,\beta \right]$ and gH differentiable in $(\alpha,\beta ) $,then $ \hat{f}(\beta)\ominus_{gH} \hat{f}(\alpha)\subset \hat{f}^{\prime}(\Delta)(\beta-\alpha) $,where $ \hat{f}^{\prime}(\Delta)=\cup_{\xi \in \Delta}\hat{f}^{\prime}(\xi)$.
\end{theorem}
\begin{theorem}[Expansion theorem of single variable interval valued function]
\label{Expansion theorem of single variable for interval valued function}
 Let $ \hat{f}:\mathbb{R}\rightarrow I(\mathbb{R}) $ be  such that $\hat{f}^{\prime},\hat{f}^{\prime \prime},\cdots, \hat{f}^{n}$ exist and $\mu-$monotonic over $\Delta$, where $\Delta =[a,x]$. Moreover consider an interval valued function $ \hat{\Phi}:\Delta\rightarrow I(\mathbb{R}) $ as\\
$ \hat{\Phi}(t)= \sum_{i=1}^{n} \hat{ \phi_{i}}(t)$, where $\hat{ \phi_{i}}(t)= \alpha_i (t) \hat{f}^{(i-1)}(t)$ with $ \alpha_{i}(t)= \frac{(x-t)^{i-1}}{(i-1)!} $ such that $ \hat{ \phi_{i}} $ is $ \mu $ monotonic for each $i$.
Then for any $x\in (a,x]$,
\begin{equation}\label{eq}
\begin{multlined}
 \hat{f}(x)\ominus_{gH}\left\lbrace \hat{f}(a)\oplus (x-a)\hat{f}^{'}(a) \oplus \frac{(x-a)^{2}}{2!}\hat{f}^{''}(a) \oplus \cdots    \oplus \frac{(x-a)^{n-1}}{(n-1)!}\hat{f}^{n-1}(a)  \right \rbrace \\  \subset\cup_{\theta \in [0,1]}\frac{(x-a)^{n}(1-\theta)^{n-1}}{(n-1)!}\hat{f}^{n}(a+\theta(x-a))
 \end{multlined}
\end{equation}
\end{theorem}
\begin{proof}
In explicit form, $ \hat{\Phi}(t) $ can be written as
\begin{equation}\label{eq 1}
\hat{\Phi}(t)=\hat{f}(t) \oplus (x-t)\hat{f}^{\prime}(t) \oplus \frac{(x-t)^{2}}{2!}\hat{f}^{\prime \prime}(t)\oplus \cdots \oplus \frac{(x-t)^{n-1}}{(n-1)!}\hat{f}^{n-1}(t)
\end{equation}
Here $ \hat{f},\hat{f}^{\prime},\cdots,\hat{f}^{n-1},\hat{f}^{n} $ exist and $ \mu $ monotonic over $\Delta$.  $ \alpha_i (t) $ is differentiable in $ \Delta $, so using Proposition \ref{lma}, $\hat{\phi}_{i}(t) $ is differentiable for each i. Hence differentiability of $ \hat{\Phi} $ in  $nbd(a)$ follows from Theorem \ref{Theorem 7}.
Here two possible cases may arise.\\
\textbf{Case 1}.  Suppose for each $ i $, $ \hat{\phi_{i}}(t)$  are  equally $ \mu- $ monotonic in $ \Delta $. Assume that each $ \hat{\phi_{i}}(t)$ is $ \mu $ increasing.
In particular let $ n $ be even and $ n=2 $. From \eqref{eq 1},\\
$ \hat{\Phi}(t)= \hat{f}(t) \oplus (x-t)\hat{f}^{\prime}(t) $.
From Proposition \ref{lma} and Theorem \ref{Theorem 7},
\begin{align*}
\hat{\Phi}^{\prime}(t)&= [\underline{f}^{\prime}(t),\underline{f}^{\prime}(t)] \oplus [(x-t)\underline{f}^{\prime \prime}(t)-\underline{f}^{\prime}(t) , (x-t)\overline{f}^{\prime \prime}(t) -\overline{f}^{\prime}(t) ]\\
&= [(x-t)\underline{f}^{\prime \prime}(t), (x-t)\overline{f}^{\prime \prime}(t)] \\
&=\frac{(x-t)^{2-1}}{(2-1)!}\hat{f}^{\prime \prime}(t)
\end{align*}
Let $n$ be odd and $ n=3$. From \eqref{eq 1},
$ \hat{\Phi}(t)= \hat{f}(t) \oplus (x-t)\hat{f}^{\prime}(t) \oplus \frac{(x-t)^2}{2!} \hat{f}^{\prime \prime}(t) $.\\
From Proposition \ref{lma} and Theorem \ref{Theorem 7},
\begin{align*}
\hat{\Phi}^{\prime}(t)&= [\underline{f}^{\prime}(t),\underline{f}^{\prime}(t)] \oplus [(x-t)\underline{f}^{\prime \prime}(t)-\underline{f}^{\prime}(t) , (x-t)\overline{f}^{\prime \prime}(t) -\overline{f}^{\prime}(t) ]\\ & \oplus \left[ \frac{(x-t)^2}{2!} \underline{f}^{(3)}(t)-(x-t)\underline{f}^{\prime \prime}(t), \frac{(x-t)^2}{2!} \overline{f}^{(3)}(t)-(x-t)\overline{f}^{\prime \prime}(t)\right]\\
&= \frac{(x-t)^{(3-1)}}{(3-1)!}\hat{f}^{(3)}(t)
\end{align*}
In general,
$ \hat{\Phi}^{\prime}(t)=\frac{(x-t)^{(n-1)}}{(n-1)!}\hat{f}^{(n)}(t)  $.\\
Similar result can  be derived if all $ \hat{\phi_{i}}(t)$  are equally $ \mu $ decreasing.\\ \\
\textbf{Case-2} Assume that $ \hat{\phi_{i}}(t) $s are differently $ \mu $ monotonic. In that case for at least two consecutive  $ \hat{\phi_{i}}(t) $s, one is $ \mu $ decreasing and another is $ \mu $ increasing.
Let  $ n=2 $ and $ \hat{\phi_{1}}(t) $ is $ \mu $ increasing and $ \hat{\phi_{2}}(t) $ is $ \mu $ decreasing. Then from Proposition \ref{lma} and Theorem \ref{Theorem 7},
\begin{align*}
\hat{\Phi^{\prime}}(t)&=[\underline{f}^{\prime}(t),\underline{f}^{\prime}(t)] \ominus_{gH} \left\lbrace \ominus_{gH} [(x-t)\overline{f}^{\prime \prime}(t) -\overline{f}^{\prime}(t), (x-t)\underline{f}^{\prime \prime}(t)-\underline{f}^{\prime}(t)] \right \rbrace \\
&= [\underline{f}^{\prime}(t),\underline{f}^{\prime}(t)] \ominus_{gH} [\underline{f}^{\prime}(t)-(x-t)\underline{f}^{\prime \prime}(t), \overline{f}^{\prime}(t)-(x-t)\overline{f}^{\prime \prime}(t)] \\
&=[(x-t)\overline{f}^{\prime \prime}(t),(x-t)\underline{f}^{\prime \prime}(t)]\\
&=(x-t)\hat{f}^{\prime \prime}(t)
\end{align*}
Let  $ n=3 $ and $ \hat{\phi_{1}}(t) $, $ \hat{\phi_{3}}(t) $ are $ \mu $ increasing and $ \hat{\phi_{2}}(t) $ is $ \mu $ decreasing. Then using Proposition \ref{lma} and Theorem \ref{Theorem 7},
\begin{align*}
\hat{\Phi^{\prime}}(t)&=[\underline{f}^{\prime}(t),\underline{f}^{\prime}(t)] \ominus_{gH} \left\lbrace \ominus_{gH} [(x-t)\overline{f}^{\prime \prime}(t) -\overline{f}^{\prime}(t), (x-t)\underline{f}^{\prime \prime}(t)-\underline{f}^{\prime}(t)] \right \rbrace \\ &
\oplus \left[ \frac{(x-t)^2}{2!} \underline{f}^{(3)}(t)-(x-t)\underline{f}^{\prime \prime}(t), \frac{(x-t)^2}{2!} \overline{f}^{(3)}(t)-(x-t)\overline{f}^{\prime \prime}(t)\right]\\
&= \frac{(x-t)^{(3-1)}}{(3-1)!}\hat{f}^{(3)}(t)
\end{align*}
In general, one can write,
$ \hat{\Phi}^{\prime}(t)=\frac{(x-t)^{(n-1)}}{(n-1)!}\hat{f}^{(n)}(t)  $.
From Theorem \ref{Mean-Value theorem for interval functions},  \begin{equation}\label{eq 4}
\hat{\Phi}(x)\ominus_{gH}\hat{\Phi}(a)\subset (x-a)\cup_{t \in \Delta}\hat{\Phi}^{\prime}(t)=\cup_{\theta \in [0,1]}\frac{(1-\theta)^{n-1}(x-a)^{n}}{(n-1)!}\hat{f}^{n}(a+\theta(x-a))
\end{equation}
That is,
\begin{equation*}
\begin{multlined}
 \hat{f}(x)\ominus_{gH} \left\lbrace \hat{f}(a) \oplus (x-a)\hat{f}^{'}(a)\oplus \frac{(x-a)^{2}}{2!}\hat{f}^{''}(a)\oplus \cdots    \oplus \frac{(x-a)^{n-1}}{(n-1)!}\hat{f}^{n-1}(a) \right \rbrace \\ \subset \cup_{\theta \in [0,1]}\frac{(x-a)^{n}(1-\theta)^{n-1}}{(n-1)!}\hat{f}^{n}(a+\theta(x-a))
 \end{multlined}
\end{equation*}
Hence the theorem.
\end{proof}
\begin{corollary}
\label{remainder}
Suppose there exists $k>0$ and $ M  > 0 $, such that for n sufficiently large,\\ $\Vert \hat{f}^{(n)}(x)\Vert < kM^{n}~\forall~x\in nbd(a) $. Then $(\frac{(x-a)^{n}(1-\theta)^{n-1}}{(n-1)!})\cup_{\theta \in [0,1]}\hat{f}^{n}(a+\theta(x-a))\rightarrow \hat{0} $ as $n\rightarrow\infty $.
\end{corollary}
\begin{proof}
$ \Vert(\frac{(x-a)^{n}(1-\theta)^{n-1}}{(n-1)!})\hat{f}^{n}(\xi)\Vert\leqslant \frac{\mid x-a\mid^{n}(1-\theta)^{n-1}}{(n-1)!}kM^{n} $ holds for any $ \xi\in nbd(a) $. \\
 $\lim_{n\rightarrow\infty}\frac{M^{n-1}\mid x-a\mid^{n-1}}{(n-1)!}=0$ and $\lim_{n\rightarrow\infty}(1-\theta)^{n-1}=
\begin{cases}
0 & \text{$\theta \neq 0$}\\
1 & \text{$\theta = 0$}
\end{cases}$.\\
This implies
 $ (\frac{(x-a)^{n}(1-\theta)^{n-1}}{(n-1)!})\hat{f}^{n}(\xi)\rightarrow \hat{0} ~~as ~~n\rightarrow\infty $ for each $ \xi\in nbd(a)  $ and hence\\
 $(\frac{(x-a)^{n}(1-\theta)^{n-1}}{(n-1)!})\cup_{\theta \in [0,1]}\hat{f}^{n}(a+\theta(x-a))\rightarrow \hat{0} $ as $n\rightarrow\infty $.
\end{proof}
If the condition of Corollary \ref{remainder} is satisfied in Theorem \ref{Expansion theorem of single variable for interval valued function} for sufficiently large $ n $, then
\begin{align*}
\hat{f}(x)\ominus_{gH}\left\lbrace \hat{f}(a) \oplus (x-a)\hat{f}^{'}(a)\oplus \frac{(x-a)^{2}}{2!}\hat{f}^{''}(a)\oplus \cdots    \oplus \frac{(x-a)^{n-1}}{(n-1)!}\hat{f}^{n-1}(a) \right \rbrace  \rightarrow \hat{0}
\end{align*}
Hence
\begin{equation}\label{eq7}
 \hat{f}(x)\approx  \hat{f}(a) \oplus (x-a)\hat{f}^{'}(a)\oplus \frac{(x-a)^{2}}{2!}\hat{f}^{''}(a)\oplus \cdots    \oplus \frac{(x-a)^{n-1}}{(n-1)!}\hat{f}^{n-1}(a)
\end{equation}
\begin{ex}\label{EE}
Consider the expansion of  $ \hat{f}(x)= e^{[-1,2]x}= \begin{cases}
  [\exp(-x),\exp(2x)],~~ if~~ x \geqslant 0\\
   [\exp(2x),\exp(-x)], ~~if ~~ x < 0
\end{cases} $ about $ a  =1 $.\\
For $ x \geq 0 $, $ \mu_{\hat{f}(x)}=\exp(2x)-\exp(-x)$. $ \mu^{\prime}_{\hat{f}}(x)=2\exp(2x)+\exp(-x)>0~~\forall~~x$.\\
Therefore $ \hat{f}(x) $ is $ \mu $-increasing and also differentiable and $ \hat{f}^{\prime}(x)=[\underline{f}^{\prime}(x), \overline{f}^{\prime}(x)] =[-\exp(-x),2\exp(2x)]$.\\
$\mu_{\hat{f}^{\prime}}^{\prime}(x)=2^{2}\exp(2x)-\exp(-x)>0 ~~for~~x\geq 0$.\\
Therefore $ \hat{f}^{\prime}(x) $ is $ \mu $-increasing and also differentiable and $\hat{f}^{\prime\prime}(x)=[\exp(-x),4\exp(2x)] $.\\
Proceeding in a similar way $ \hat{f}^{(n)}(x)=[(-1)^{n}\exp(-x),2^{n}\exp(2x)] $  which is $\mu-$increasing $\forall~~n$.\\ For $ \xi\in [1,x],\Vert\hat{f}^{(n)}(\xi)\Vert\leq 2^{n}\exp(2x) $. Now $ \lim_{n\rightarrow\infty}\frac{(x-1)^{n}2^{n}}{(n-1)!}=0 $. Hence $ \Vert\hat{f}^{(n)}(\xi)\Vert \rightarrow 0  $ as $ n\rightarrow\infty $. Therefore all conditions of Theorem  \ref{Expansion theorem of single variable for interval valued function} and Corollary \ref{remainder} hold at $a=1$. Hence expansion of $ \hat{f}(x) $ in \eqref{eq7} about $ a=1 $ becomes
\begin{align*}
[\exp(-x),\exp(2x)]&\approx [\exp(-1),\exp(2)] \oplus  (x-1)[-\exp(-1),2\exp(2)] \oplus \frac{(x-1)^{2}}{2}[\exp(-1),4\exp(2)].
 \end{align*}
\end{ex}
\subsection{Expansion of interval valued function over $\mathbb{R}^n$}
\begin{theorem}[Expansion theorem of n variable for interval valued function]
\label{Expansion theorem of n variable for interval valued function}
Let $ \hat{f}: \Omega \subseteq \mathbb{R}^{n}\rightarrow I(\mathbb{R}) $ be gH differentiable up to order $ s $ on open convex subset $ \Omega $ of $\mathbb{R}^{n}  $ and $ \hat{f} $ and all the partial derivatives of $ \hat{f} $ up to order s are component-wise $ \mu $-monotonic over $ \Omega $ . Moreover for any $ \xi \in [0,1] $ if there exists an interval valued function $ \hat{\Psi}:[0,1]\rightarrow I(\mathbb{R}) $ as $ \hat{\Psi}(t)= \sum_{i=1}^{n} \hat{ \psi_{i}}(t)$, where $  \hat{ \psi_{i}}(t)= \frac{(\xi -t)^{i-1}}{(i-1)!}  \hat{g}^{(i-1)}(t)$,  $ \hat{g}(t)=\hat{f}(\gamma(t)) $ with $ \gamma(t)= a+tv,~ v=x-a $ for $ a,x \in \Omega $, $ t \in [0,1] $ such that $ \hat{ \psi_{i}} $ is $ \mu $ monotonic for each $i$. Then
\begin{equation}\label{eq8}
\begin{multlined}
\hat{f}(x)\ominus_{gH}\left\lbrace \hat{f}(a) \oplus \left(  \sum_{i=1}^{n}\frac{\partial \hat{f}(a)}{\partial x_{i}}(x_{i}-a_{i}) \right) \oplus \left( \frac{1}{2!}\sum_{i,j=1}^{n}\frac{\partial^{2}\hat{f}(a)}{\partial x_{i}\partial x_{j}}(x_{i}-a_{i})(x_{j}-a_{j}) \right) \oplus  \cdots \right.\\\left.\oplus \cdots  \left( \frac{1}{(s-1)!}\sum_{i_{1},i_{2},...,i_{s}=1}^{n} \frac{\partial^{s-1}\hat{f}(a)}{\partial x_{i1}...\partial xi_{s-1}}(x_{i1}-a_{i1})...(x_{is-1}-a_{is-1})\right)  \right\rbrace \\
\subset \cup_{c \in L.S\left\lbrace a,x \right\rbrace }\sum_{i_{1},i_{2},...,i_{s}=1}^{n}\frac{1}{(s-1)!}\frac{\partial^{s}\hat{f}(c)}{\partial x_{i1}...\partial x_{is}}(x_{i1}-a_{i1})...(x_{is}-a_{is}),
 \end{multlined}
\end{equation}
  where $L.S\left\lbrace a,x \right\rbrace  $ is the line segment joining $a$ and $x $.
\end{theorem}
\begin{proof}
 $ \hat{f}: \Omega \rightarrow I(\mathbb{R}) $ and $ \gamma: [0,1] \rightarrow \Omega $. Since $ \Omega $ is a convex subset of $ \mathbb{R}^{n} $, for $ a, b \in \Omega $, $ a+t(b-a) $ with $ t \in [0,1] $  must belongs to $ \Omega $. $\hat{g}:[0,1] \rightarrow I(\mathbb{R})$ is defined by $ \hat{g}(t)= \hat{f}( \gamma_1(t), \gamma_2(t), \cdots \gamma_n(t)) $, where $ \gamma_i(t)= a_i+t (b_i-a_i), ~\forall~  i=1,2, \cdots, n $, $ t \in [0,1] $. By Corollary \ref{comp},  $ \hat{g} $ is differentiable. Since $ \hat{f} $  gH differentiable up to order $s$ so $ \hat{g} $ is also differentiable up to order $s$. Hence $\hat{ \Psi}(t) $ exists. Therefore \\
$ \hat{g}^{\prime}(t)= \sum_{i=1}^{n} \gamma_i^{\prime}(t) \frac{\partial \hat{f} (\gamma(t))}{\partial x_{i}}= \sum_{i=1}^{n}(b_i-a_i) \frac{\partial \hat{f}(\gamma(t))}{\partial x_{i}}=  \nabla \hat{f}(\gamma(t))^Tv  $. Hence
$\hat{g}^{\prime\prime}(t)=
v^{T}\nabla^{2}(\hat{f}(\gamma(t))v$.\\
By induction, $ \hat{g}^{(s)}(t)=\sum_{i_{1},i_{2},\cdots,i_{s}=1}^{n}\frac{\partial^{s} \hat{f}(\gamma (t))}{\partial x_{i1}\partial x_{i2}\cdots \partial x_{is}} (x_{i1}-a_{i1})\cdots(x_{is}-a_{is})       $\\
From the assumptions of the theorem, $ \hat{\Psi} $ satisfies all the condition of Theorem \ref{Expansion theorem of single variable for interval valued function}. Using the expression of Theorem \ref{Expansion theorem of single variable for interval valued function} about `$ 0 $',

\begin{equation}\label{eq9}
 \hat{g}(t)\ominus_{gH}\left\lbrace \hat{g}(0)\oplus t \hat{g}^{\prime}(0) \oplus \frac{t^{2}}{2!}\hat{g}^{\prime \prime}(0) \oplus  \cdots \oplus \frac{t^{(s-1)}}{(s-1)!}\hat{g}^{(s-1)}(0)\right\rbrace \subset \cup_{\theta \in [0,1]}\frac{t^{s}}{(s-1)!}\hat{g}^{(s)}(\theta),
\end{equation}
In particular for t=1,
\begin{equation}\label{eq10}
\hat{g}(1)\ominus_{gH}\left\lbrace \hat{g}(0)\oplus  \hat{g}^{\prime}(0) \oplus  \frac{1}{2!}\hat{g}^{\prime \prime}(0)\oplus \cdots \oplus \oplus \frac{1}{(s-1)!}\hat{g}^{(s-1)}(0)\right\rbrace  \subset \cup_{\theta \in [0,1]}\frac{1}{(s-1)!}\hat{g}^{(s)}(\theta).
\end{equation}
$\hat{g}(1)=\hat{f}(x),\hat{g}(0)=\hat{f}(a)$, $ \hat{g}^{\prime}(0)=\sum_{i=1}^{n}\frac{\partial\hat{f}(a)}{\partial x_{i}}(x_{i}-a_{i})$,  $\hat{g}^{\prime\prime}(0)=\sum_{i,j=1}^{n}\frac{\partial^{2}\hat{f}(a)}{\partial x_{i}\partial x_{j}}(x_{i}-a_{i})(x_{j}-a_{j})  $, etc..\\
\eqref{eq8} follows after substituting these values in \eqref{eq10}.
\end{proof}
\begin{corollary}\label{cor}
Suppose there exist $ k>0$ and $ M  > 0 $, such that for sufficiently large n, \\ $\Vert\frac{\partial^{s}\hat{f}(c)}{\partial x_{i1}...\partial x_{is}} \Vert < kM^{s}~\forall~ c\in  L.S \left\lbrace  a , x\right\rbrace $. Then $$ \cup\sum_{i_{1},i_{2},...,i_{s}=1}^{n}(\frac{1}{(s-1)!})\frac{\partial^{s}\hat{f}(c)}{\partial x_{i1}...\partial x_{is}}(x_{i1}-a_{i1})...(x_{is}-a_{is})\rightarrow \hat{0}~~ as~~ s\rightarrow\infty $$
\end{corollary}
\begin{proof}For any c $ \in L.S \left \lbrace a,x \right\rbrace$,
$\|\frac{\partial^{s}\hat{f}(c)}{\partial x_{i1}...\partial x_{is}}\|=\max\left\lbrace \mid\frac{\partial^{s}\underline{f}(c)}{\partial x_{i1}...\partial x_{is}}\mid,\mid \frac{\partial^{s}\overline{f}(c)}{\partial x_{i1}...\partial x_{is}}\mid\right\rbrace $.\\
From Corollary \ref{remainder},
$$\left( \frac{1}{(s-1)!}\right) \sum_{i_{1},i_{2},...,i_{s}=1}^{n}\frac{\partial^{s}\hat{f}(c)}{\partial x_{i1}...\partial x_{is}}(x_{i1}-a_{i1})...(x_{is}-a_{is})         \rightarrow \hat{0} ~~as~~s\rightarrow\infty $$
Therefore $ \cup(\frac{1}{(s-1)!})\sum_{i_{1},i_{2},...,i_{s}=1}^{n}\frac{\partial^{s}\hat{f}(c)}{\partial x_{i1}...\partial x_{is}}(x_{i1}-a_{i1})...(x_{is}-a_{is})\rightarrow \hat{0}~~ as~~ s\rightarrow\infty $.
\end{proof}
Following result holds as a consequence of \eqref{eq8} and Corollary \ref{cor}.
\begin{equation}\label{eq11}
\begin{multlined}
 \hat{f}(x)\approx  \hat{f}(a) \oplus \left(  \sum_{i=1}^{n}\frac{\partial \hat{f}(a)}{\partial x_{i}}(x_{i}-a_{i})\right)  \oplus \left( \frac{1}{2!}\sum_{i,j=1}^{n}\frac{\partial^{2}\hat{f}(a)}{\partial x_{i}\partial x_{j}}(x_{i}-a_{i})(x_{j}-a_{j})\right)  \oplus  \cdots \\ \oplus \cdots \left(  \frac{1}{(s-1)!}\sum_{i_{1},i_{2},...,i_{s}=1}^{n} \frac{\partial^{s-1}\hat{f}(a)}{\partial x_{i1}...\partial xi_{s-1}}(x_{i1}-a_{i1})...(x_{is-1}-a_{is-1}) \right)
 \end{multlined}
\end{equation}
%
\begin{ex}\label{EF}
Consider $ \hat{f}(x_{1},x_{2})=[-2,3]x_{1}e^{[-1,2]x_{2}} $.\\
$ \hat{f}(x_{1},x_{2})= [\underline{f}( x_{1},x_{2} , \overline{f}( x_{1},x_{2} ]=\begin{cases}
\left[ -2x_{1}e^{2x_{2}},3x_{1}e^{2x_{2}}\right] ,& \text{ if } x_{1}\geq 0 , x_{2} \geq 0\\
\left[ 3x_{1}e^{2x_{2}},-2x_{1}e^{2x_{2}}\right], & \text{ if } x_{1}\leq 0 , x_{2}\geq 0\\
\left[ 3x_{1}e^{-x_{2}},-2x_{1}e^{-x_{2}}\right] ,& \text{ if } x_{1}\leq 0 , x_{2} \leq 0\\
\left[ -2x_{1}e^{-x_{2}},3x_{1}e^{-x_{2}}\right] & \text{ if } x_{1}\geq 0 , x_{2} \leq 0
\end{cases}     $.\\
Consider the quadratic expansion of $ \hat{f}(x_{1}, x_{2})= \left[ -2x_{1}e^{2x_{2}},3x_{1}e^{2x_{2}}\right]$ , $ x_{1}\geq 0 , x_{2} \geq 0 $ about $ a= (2,2) $.\\
$ \mu_{\hat{f}}(x_{1},x_{2})=3x_{1}e^{2x_{2}}+ 2x_{1}e^{2x_{2}}$. From the derivative of $ \mu_{\hat{f}}(x_{1},x_{2})$
it can be easily verified that
\begin{enumerate}[(i)]
\item  $ \hat{f} $  is $ \mu $-increasing with respect to $ x_1 $ and $ x_2 $ both,
\item  $  \frac{\partial \hat{f}}{\partial x_{1}} $ is  $\mu $ -increasing with respect to $ x_{2} $ and , $  \frac{\partial \hat{f}}{\partial x_{2}} $ is $\mu $-increasing with respect to $ x_1 $ and $ x_{2} $ both.
\end{enumerate}
Using \eqref{eq11}, quadratic expansion of $ \hat{f}(x_{1},x_{2}) $ about $ (2,2) $ becomes
\begin{align*}
 \hat{f}(x_{1},x_{2})\approx [-4e^{4},6e^{4}]\oplus \left\lbrace (x_{1}-2)[-2e^{4} ,3e^{4}] \oplus (x_{2}-2)[-8e^{4},12e^{4}]\right\rbrace \\
 \oplus \left\lbrace  (x_{1}-2)(x_{2}-2)[-4e^{4},6e^{4}] \oplus \frac{(x_{2}-2)^{2}}{2}[-16e^{4},24e^{4}]\right\rbrace
\end{align*}

\end{ex}

\section{Conclusion and future scope}
In this article calculus of interval valued function is discussed using $ \mu $-monotonic property and composite mapping of interval valued function and real valued function is studied. Expansions of interval valued function over $ \mathbb{R} $ and $ \mathbb{R}^n $ are developed using composite mapping and gH diffentiability.  This expansion can provide a powerful tool for developing algorithms for solution of system of equation, least mean square problems with interval parameters, which may be considered as the future scope of the present contribution.


\begin{thebibliography}{00}

\bibitem{aubindifferential}
JP~Aubin and A~Cellina,
 {\it Differential Inclusions}.
(Springer-Verlag. Springer, 1984).

\bibitem{chalco2011generalized}
Yurilev Chalco-Cano, Heriberto Rom{\'a}n-Flores and Mar{\'\i}a-Dolores
  Jim{\'e}nez-Gamero,
 Generalized derivative and $\pi$-derivative for set-valued functions,
{\it Information Sciences}. {\bf 181}
(2011) 2177--2188.

\bibitem{de2007existence}
FS~De~Blasi, V~Lakshmikantham and T~Gnana Bhaskar,
An existence theorem for set differential inclusions in a semilinear
  metric space,
{\it Control and Cybernetics} {\bf 36}(2007) 571.

\bibitem{galanis2005set}
GN~Galanis, T~Gnana Bhaskar, V~Lakshmikantham and PK~Palamides,
Set valued functions in fr{\'e}chet spaces: Continuity, hukuhara
  differentiability and applications to set differential equations,
 {\it Nonlinear Analysis: Theory, Methods \& Applications}
  {\bf 61}(2005) 559--575.

\bibitem{ibrahim1996differentiability}
AG~M Ibrahim,
 On the differentiability of set-valued functions defined on a banach
  space and mean value theorem,
{\it Applied mathematics and computation} {\bf 74}(1996) 79--94.

\bibitem{li2009calculus}
SJ~Li, KW~Meng and J-P Penot, Calculus rules for derivatives of multimaps,
{\it Set-Valued and Variational Analysis} {\bf 17}(2009) 21--39.

\bibitem{tu2009stability}
Nguyen~Ngoc Tu and Tran~Thanh Tung, Stability of set differential equations and applications,
{\it Nonlinear Analysis: Theory, Methods \& Applications},
  {\bf 71}(2009) 1526--1533.

\bibitem{bede2013generalized}
Barnab{\'a}s Bede and Luciano Stefanini, Generalized differentiability of fuzzy-valued functions,
{\it Fuzzy Sets and Systems} {\bf 230}(2013) 119--141.
%
\bibitem{chalco2013calculus} Yurilev Chalco-Cano, Antonio Rufi{\'a}n-Lizana, Heriberto Rom{\'a}n-Flores, and
  Mar{\'\i}a-Dolores Jim{\'e}nez-Gamero, Calculus for interval-valued functions using generalized hukuhara
  derivative and applications, {\it Fuzzy Sets and Systems} {\bf 219}(2013) 49--67.
%
\bibitem{costa2015generalized} TM~Costa, Yurilev Chalco-Cano, Weldon~A Lodwick and Geraldo~Nunes Silva, Generalized interval vector spaces and interval optimization,
{\it Information Sciences}{\bf 311}(2015) 74--85.
%
\bibitem{Lupulescu201350}
Vasile Lupulescu, Hukuhara differentiability of interval-valued functions and interval
  differential equations on time scales,
{\it Information Sciences} {\bf 248}(2013) 50 -- 67.
%
\bibitem{Malinowski2011JAML}
  M.T.~Malinowski, Interval differential equations with a second type hukuhara
  derivative,
  {\it Applied Mathematics Letters}. {\bf 24}(2011) 2118--2123.

\bibitem{stefanini2008generalization}
Luciano Stefanini, A generalization of hukuhara difference for interval and fuzzy
  arithmetic,
{\it Soft Methods for Handling Variability and Imprecision}, Series on Advances in Soft Computing. Vol.~48 (2008).
%
\bibitem{stefanini2009generalized}
Luciano Stefanini and Barnabas Bede, Generalized hukuhara differentiability of interval-valued functions
  and interval differential equations,
{\it Nonlinear Analysis: Theory, Methods \& Applications}.
  {\bf 71}(2009) 1311--1328.
%
\bibitem{rall1983mean}
Louis~B Rall, Mean value and taylor forms in interval analysis,
{\it SIAM Journal on Mathematical Analysis}.{\bf 14}(1983) 223--238.
%
\bibitem{caprani1980mean}
Ole Caprani and Kaj Madsen, Mean value forms in interval analysis,
{\it Computing}. {\bf 25}(1980) 147--154.
%
\bibitem{hansen2003global}
Eldon Hansen and G~William Walster,
 {\it Global optimization using interval analysis: revised and
  expanded}. Vol.~264 (CRC Press, 2003).
%
\bibitem{Baker1996kap}
R~Baker Kearfott and Vladik Kreinovich,
{\it Applications of interval computations}. Vol.~3 (Springer Science \& Business Media, 2013).
%
\bibitem{moore2009introduction}
Ramon~E Moore, R~Baker Kearfott and Michael~J Cloud.
{\it Introduction to interval analysis}, (Siam, 2009).
%
\bibitem{Neumaier1990book}
Arnold Neumaier,
 {\it Interval methods for systems of equations}, Vol.~37.
 (Cambridge university press, 1990).
%
\bibitem{markov1979calculus}
Svetoslav Markov, Calculus for interval functions of a real variable,
 {\it Computing}, {\bf 22}(1979) 325--337.

\bibitem{osuna2015optimality}
R~Osuna-G{\'o}mez, Yurilev Chalco-Cano, Beatriz Hern{\'a}ndez-Jim{\'e}nez and
  G~Ruiz-Garz{\'o}n. Optimality conditions for generalized differentiable interval-valued
  functions,
 {\it Information Sciences}, {\bf 321}(2015) 136--146.


\bibitem{markov1977extended}
SM~Markov, Extended interval arithmetic,
 {\it CR Acad. Bulgare Sci}, {\bf 30}(1977) 1239--1242.

\end{thebibliography}

\end{document}